\documentclass[11pt, letterpaper]{article}
\usepackage[margin=2.5cm]{geometry}
\usepackage{helvet}
\usepackage{courier}
\usepackage{url}
\usepackage{bm}
\usepackage{graphicx}
\usepackage{tcolorbox}
\usepackage[dvipsnames]{xcolor}
\usepackage{enumerate}
\usepackage{pgfplots}
\usepackage{algorithm}
\usepackage{eucal}
\usepackage{mathtools}
\usepackage{float}
\usepackage{amsthm}
\usepackage{amsmath,amssymb,amsfonts}
\usepackage{xcolor}
\usepackage{bbold}
\usepackage{xspace}
\usepackage{lettrine}
\usepackage{subfigure}
\usepackage{tikz}
\usetikzlibrary{backgrounds,automata}
\usetikzlibrary{shapes.geometric}
\usepackage{latexsym}
\usepackage{fancybox}
\usepackage{setspace}
\usepackage{graphicx}
\usepackage{thmtools}

\usepackage{hyperref}
\hypersetup{
    colorlinks=true,    % Colors the text of links instead of showing colored boxes
    citecolor=blue,     % Color of citation links (e.g., [1] or Smith et al., 2020)
    linkcolor=blue,     % Color of internal links (sections, figures, tables)
    urlcolor=blue       % Color of web links
}

\usepackage[capitalise]{cleveref}

\theoremstyle{plain}
\newtheorem{theorem}{Theorem}[section]
\newtheorem{lemma}[theorem]{Lemma}
\newtheorem{definition}[theorem]{Definition}

\newtheorem{remark}[theorem]{Remark}
\newtheorem{corollary}[theorem]{Corollary}

\newtheorem{example}[theorem]{Example}

\Crefname{theorem}{Theorem}{Theorems}
\Crefname{lemma}{Lemma}{Lemmas}
\Crefname{definition}{Definition}{Definitions}
\Crefname{corollary}{Corollary}{Corollaries}
\usepackage[backend=biber, maxbibnames=99]{biblatex}
\usepackage{algpseudocode}
\newcommand{\PG}{\mathrm{PG}}

\begin{document}
\title{Explicit strong blocking sets and minimal codes from expander graphs using a new graph parameter}
\author{Anurag Bishnoi\thanks{Delft University of Technology, Netherlands. \emph{E-mail}: \text{a.bishnoi@tudelft.nl}.}
\and 
Ioannis Kakatelis \thanks{Czech Technical University in Prague, Czech Republic. \emph{E-mail}: \text{ioannis.kakatelis@fit.cvut.cz}.}
}
\date{}

\maketitle

\begin{abstract}
    A strong blocking set is a set of points in a finite projective space that intersects every hyperplane of the space in a spanning set of the hyperplane. 
    These objects have been studied both in finite geometry and coding theory due to their equivalence with minimal codes. 
    We give an explicit construction of strong blocking sets using expander graphs that improves the explicit construction of Alon et al.~\cite{ABDN2023} by introducing a new geometrically motivated graph parameter that is better suited for this problem.
    Using this approach, we improve the constant $C$ in the explicit construction of a strong blocking set of size $C (q+1) k$ in the $(k - 1)$-dimensional projective space over $\mathbb{F}_q$ for every prime power $q \geq 7$.
\end{abstract}

\section{Introduction}
Strong blocking sets in finite projective spaces have been studied extensively in the last ten years.
These objects were first introduced by Davydov et al.~\cite{davydov2011linear} to construct saturating sets and then reintroduced by H\'eger et al.~\cite{heger2015search} to study the $q$-analog of a classic combinatorial search problem: identifying an unknown projective point using subspace queries. 
Since then, these objects have been reinvented and studied under various names, such as cutting blocking sets \cite{alfarano2022geometric,bartoli2022cutting, bonini2021minimal}, generating sets \cite{fancsali2014lines}, and the now commonly used term of strong blocking sets \cite{ABDN2023, heger2021short, lia2025note}. 

The recent flurry of results is primarily motivated by a connection with coding theory that was independently discovered by Alfarano et al.~\cite{alfarano2022geometric} and Tang et al.~\cite{tang2021full}. 
They showed that a strong blocking set of size $n$ in a finite projective space $\mathrm{PG}(k - 1, q)$ is equivalent to a minimal $k$-dimensional linear code of length $n$. 
Minimal codes are linear codes in which the support of a nonzero codeword $u$ is a subset of the support of a nonzero codeword $v$ if and only if $u = \lambda v$ for some scalar $\lambda$. 
These codes have their own history dating back to 1993 in the work of Massey on secret sharing protocols \cite{massey1993minimal}. 
Then Ashikhmin and Barg established sufficient conditions for their existence \cite{ashikhmin1998minimal} which led to several interesting constructions. 
In 2019, it was shown that these objects are in fact equivalent to strong blocking sets, thus converging these two threads of research. 
Since then, strong blocking sets and minimal codes have also found connections with perfect hashing \cite{Bishnoi2023trifference} and additive combinatorics \cite{borello2025geometry}.

In \cite{alfarano2022three}, it was shown that the smallest length $n$ of a minimal code of dimension $k$, or equivalently the smallest size $n$ of a strong blocking set in $\PG(k - 1, q)$, must be at least $(q + 1)(k - 1)$. This lower bound was further improved in \cite{Bishnoi2023trifference} and \cite{scotti2024lower}.
A probabilistic argument \cite{Bishnoi2023trifference, heger2021short} shows the existence of such codes of length at most $2(q + 1)k$.
One of the central problems on minimal codes is to give \textit{explicit constructions} of such codes that have length at most $C q k$ for some constant $C$. 
The main regime of interest is fixed, but possibly large, $q$ and $k \rightarrow \infty$. 
The first author solved this problem in collaboration with Alon, Das and Neri~\cite{ABDN2023} by giving a graph-theoretic construction that combined explicit construction of asymptotically good error-correcting codes with explicit constant degree expander graphs. While for $q \leq 5$, the best constant is obtained via earlier concatenation based constructions \cite{bartoli2023small}, for all $q \geq 7$, the graph-theoretic construction is significantly better.

\textit{In this paper}, we further optimize the graph-theoretic construction by following a similar general approach of \cite{ABDN2023} but using a better graph parameter that is fine-tuned for geometrical constraints.
This gives us a better constant $C = C(q)$ for every $q \geq 7$, summarized in Table~\ref{tab:my_table}.
For example, the best construction for $q = 7$ was of size $46.32(q + 1)k$, while our new construction is of size at most $44.61(q + 1)k$. 
The new graph parameter $f_q(G)$ for a graph $G$ incorporates the field size $q$ using the fact that through a fixed codimension-$2$ space there are exactly $q + 1$ hyperplanes. 
The previous graph parameter $\iota(G)$ used in \cite{ABDN2023}, known as the vertex integrity of a graph, did not depend on the field size. 
We show that $f_q(G) \geq \iota(G)$ for all graphs $G$ and $q \geq 2$, and we prove a spectral lower bound on $f_q(G)$ that is much better than the best known lower bound on $\iota(G)$. 
This lower bound is the key source of improvement in the final optimization.

\subsection*{Paper Outline}
In \cref{sec:prelim} we provide the necessary background and required notation to the reader about coding theory, projective spaces, strong blocking sets and expander graphs. In \cref{sec:intro_param}, we introduce the new graph parameter and prove a spectral lower bound on this parameter. In \cref{sec:construction}, we explain the construction of the strong blocking set in relation to the graph parameter.
In \cref{sec:ag_codes}, we optimize the size of strong blocking set and compare the new results with those previously derived by vertex integrity using algebraic-geometric codes, almost Ramanujan graphs and the tool of field reduction to handle non-quadratic fields. Finally, in \cref{sec:conclusion} we summarize the results of this paper and discuss how this new parameter can give rise to new research directions.

\section{Preliminaries}\label{sec:prelim}
In this section, we review the basic notions and preliminary results from coding theory that will be used throughout the paper, with particular emphasis on minimal linear codes and their geometric interpretation. We also recall the concept of expander graphs and some known results on constant degree expanders. 
Throughout the remainder of the paper, we assume that $q$ is a prime power.

\subsection{Error correcting codes and minimal codes }
We denote with $\mathbb{F}_q$ to be the finite field with of order $q$ and let $n \in \mathbb{N}$.
\begin{definition}
 The \textit{(Hamming) support} of a vector $v\in \mathbb{F}^{n}_{q}$ is the set 
 $$ \sigma(v):=\{ i \, : \, v_i\neq 0\} \subseteq [n].$$
 The \textit{(Hamming) weight} of $v$ is
 $$ wt(v):=|\sigma(v)|.$$
\end{definition}

The Hamming weight induces a  distance on $\mathbb{F}_q^n$, given by $d(u,v):= wt(u-v)$. 
This is known as the \textit{Hamming distance} and it is fundamental in the theory of error-correcting codes. 
It is the number of positions where the coordinates of two vectors differ.

\begin{definition}
An $[n,k,d]_q$ code $C$ is a $k$-dimensional subspace of $\mathbb{F}_q^n$, and
$$d \coloneqq \min\{wt(v) \, : \, v \in C \setminus\{0\}\}.$$
The parameter $d$ is called the \textit{distance} of $C$. 
The elements of $C$ are called \textit{codewords}. 
The parameters $n$ and $k$ are the length and the  dimension of the code, respectively.
A \textit{generator matrix} for $C$ is a matrix $G\in \mathbb{F}_q^{k\times n}$ such that 
$$C=\{uG\,:\, u \in \mathbb{F}_q^k\};$$
that is, the rows of $G$ span $C$.
The rate of a code $C$ is equal to $\frac{k}{n}$ and and the relative hamming distance of the code is equal to $\frac{d}{n}$.
\end{definition}

\begin{example}
Let us consider the code $C = \{000,111\} \subset \mathbb{F}_2^3$.
Then we can check that this is a $[3,1,3]_{2}$ code. 
In general, such $[n, 1, n]_q$ codes are called the repetition codes and they are always minimal.
\end{example}

The standard problem in coding theory is to study the extremal values of one parameter with respect to the others. 
This is captured by studying the asymptotic rate and relative distance as follows. 

\begin{definition}
Let $\{n_i \}_{i \geq 1}$ be a strictly increasing sequence of lengths and suppose that there exist sequences $\{k_i\}_{i \geq 1}$ and $\{d_i\}_{i \geq 1}$ such that
for all $i \geq 1$ there is an $[n_i, k_i, d_i]_q$ code $C_{i}$. 
Then the sequence $\mathbf{C} =  \{C_{i}\}_{i \geq 1}$ is called an \textit{$(R,\delta)_q$-family of codes}, where 
the \textit{rate} $R$ of $\mathbf{C}$ is defined as 
\[R \coloneqq \liminf_{i \to \infty} \frac{k_i}{n_i},\]
and the \textit{relative distance} is defined as 
\[\delta \coloneqq \liminf_{i \rightarrow \infty} \frac{d_i}{n_i}.\]
\end{definition}

One of the central problems on error-correcting codes is to understand the trade-off between the rate and the relative distance of codes. 
A family of codes $\mathbf{C}$ for which $R > 0$ and $\delta> 0$, is known as an \textit{asymptotically good code} (see \cite[Section 1.8]{coding_theory_book} or \cite[Chapter 9]{vanLint_book}). 
An easy probabilistic argument shows the existence of such codes for every $\delta \in [0, 1 - 1/q)$ and $R = 1 - H_q(\delta)$, where
\[H_q(x) \coloneqq x \log_q(q - 1) - x \log_q(x)  - (1 - x) \log_q(1 - x),\]
is the $q$-ary entropy function, defined on the domain $0 \leq x \leq 1 - 1/q$. 
This is known as the Gilbert-Varshamov bound \cite[Section 5.1]{vanLint_book}. 
The first explicit construction of asymptotically good codes was given by Justesen \cite{justesen1972class}, who showed that for every $0 < R < 1/2$, there is an explicit family of codes with rate $R$ and relative distance $\delta \geq (1 - 2R) H_q^{-1}\left(\frac{1}{2}\right)$. 
Improving the values of the rate $R$ and the relative distance $\delta$ for which there is an explicit construction, and reducing the computational complexity of these constructions, has been an active area of research in coding theory since the 1970s (see for example \cite{ta2017explicit, tsfasman1982modular}). 
One of the most significant developments in the area was the use of modular curves to show that, for square $q \geq 49$, there are explicit constructions of linear codes over $\mathbb{F}_q$ that have a better rate and relative distance than the one given by the Gilbert-Varshamov bound. 
For further details  see \cite{couvreur2021algebraic} and \cite{tsfasman2013algebraic}.

\textit{In this paper}, we study a special class of codes called \emph{minimal linear codes}. 
These codes are interesting from various different perspectives.

 Let $C$ be an $[n,k,d]_q$ code. A nonzero codeword $v \in C$ is said to be \textbf{minimal} (in $C$) if $\sigma(v)$ is minimal with respect to the inclusion in the set 
 $$ \sigma(C) \coloneqq \{ \sigma(c) \, : \, c \in C \setminus\{\vec{0}\}\}.$$

 Note that in a linear code, if $v$ is a minimal codeword then for any codeword $u$ we have $\sigma(u) \subseteq \sigma(v)$ if and only if $u = \lambda v$ for some scalar $\lambda$. 

\begin{definition}
The (linear) code $C$ over a finite field $\mathbb{F}_q$ is called a \textbf{minimal code} if all of its nonzero codewords are minimal.
Equivalently, a code is minimal if for all $u,v \in C$, we have that $\sigma(u) \subseteq \sigma(v) \iff \exists \lambda \in \mathbb{F}_q ~u = \lambda v$.
\end{definition}

\subsection{Minimal codes and strong blocking sets}
In this section, we first describe the equivalence between minimal codes and strong blocking sets, and then note the relevant prior results. 
For $k > 1$, let $\mathrm{PG}(k-1, q)$ denote the projective space of dimension $k - 1$ obtained from the vector space $\mathbb{F}_q^k$.
Its points, lines, $\ldots$, $m$-spaces, are the $1$-dimensional, $2$-dimensional, $\ldots$, $(m + 1)$-dimensional, vector subspaces of $\mathbb{F}_q^k$.  
For a nonzero vector $v \in \mathbb{F}_q^k$, let $\langle v \rangle$ be the corresponding point of $\mathrm{PG}(k-1, q)$.
For a set $S$ of points in the projective space we denote by $\langle S \rangle$ the span of these points, that is, the smallest dimensional projective subspace containing $S$.

\begin{definition}\label{def:projective_system}
A \textit{projective} $[n,k,d]_q$ \textit{system} is a (multi)set of $n$ points, $M \subseteq \mathrm{PG}(k-1,q)$, such that $M$ spans the projective space and 
$$d=n-\max\{|H\cap M| : H \text{ is a hyperplane}\}.$$
\end{definition}

The term projective system and the notation used come from the equivalence with linear codes of a given minimal distance.
A projective $[n,k,d]_q$ {system} is simply a dual interpretation of a nondegenerate $[n,k,d]_q$ code. 
We say that an $[n,k,d]_q$ code $C$ is \textit{nondegenerate} if there is no identically zero coordinate in $C$. 
The correspondence comes from putting representatives of the points of the projective system as columns of a $k\times n$ matrix, and considering the code generated by (the rows of) this matrix. Vice versa, given a generator matrix of a nondegenerate code, we can obtain the associated projective system by taking the columns of such a matrix as a multiset of points in $\mathrm{PG}(k-1,q)$. 

There is more work required to make this correspondence well-defined, and 
but it becomes so when we consider $\mathrm{PGL}(k,q)$-equivalence classes of projective $[n,k,d]_q$ systems and the so-called monomial equivalence classes of nondegenerate $[n,k,d]_q$ codes.
We refer the reader to~\cite[Theorem 1.1.6]{tsfasman2013algebraic} for a formal treatment of the correspondence. 

Due to this correspondence, a sequence $\{M_i\}_{i\in \mathbb{N}}$ of projective systems is called an \textit{$(R,\delta)_q$-family of projective systems} if the corresponding family of codes is an $(R,\delta)_q$-family of codes.
We now define the main finite geometric object studied in this paper. 

\begin{definition}
 A set $ M\subseteq \mathrm{PG}(k-1,q)$ is said to be a \textit{strong blocking set} if 
 $$\langle H\cap M\rangle=H,$$
 for every hyperplane $H$ of $\mathrm{PG}(k-1,q)$.
\end{definition}

\begin{theorem}[see \cite{alfarano2022geometric}, \cite{tang2021full}]
 Let $C$ be a nondegenerate $[n,k,d]_q$ code and let $G=(g_1 \mid \ldots \mid g_n) \in \mathbb{F}_q^{k\times n}$ be any of its generator matrices. The following are equivalent:
 \begin{enumerate}
     \item $C$ is a minimal code;
     \item $M=\{\langle g_1 \rangle,\ldots, \langle g_n \rangle\}$ is a strong blocking set in $\mathrm{PG}(k-1,q)$.
 \end{enumerate}
\end{theorem}

One of the main problem -- from both a coding theoretic and geometric point of view -- is the construction of small strong blocking sets, or, equivalently, of short minimal codes. 
The first step is to ask how small a strong blocking set can be. 
The following is a general lower bound observed in~\cite{alfarano2022geometric}, proved using the Combinatorial Nullstellensatz (see~\cite{ABDN2023, heger2021short} for alternative proofs.)

\begin{theorem}
\label{thm:strong_lower_bound}
 For any prime power $q$, every strong blocking set in $\mathrm{PG}(k - 1, q)$ has size at least  $(q+1)(k-1).$
\end{theorem}

Recently, this lower bound has been improved by using Delsarte's linear programming bound in coding theory.

\begin{theorem}[see   {\cite[Theorem 1.4]{Bishnoi2023trifference}, \cite[Theorem 3.3]{scotti2024lower}}]
    For any prime power $q$, there is a constant $c_q>1$ such that every strong blocking set in $\mathrm{PG}(k - 1, q)$ has size at least 
$(c_q-o(1))(q+1)(k-1)$.
\end{theorem}

The best known existence results for the upper bounds on the size of the smallest strong blocking set have been proved via the probabilistic method for the cases when $q = 2$ in \cite{miklos1984linear} and for $q > 2$ in \cite{Bishnoi2023trifference,ABN2023}. 

The key idea behind these constructions, as well as many explicit constructions, is to consider points sets that are unions of lines in a projective space. 
If $\ell_1, \dots, \ell_m$ is the collection of lines, then their union forms a strong blocking set (of size at most $(q + 1)m$) if and only if there is no pair $(S, H)$ where $S$ is a codimension-$2$ subspace and $H$ is a hyperplane containing $S$ such that for all $i$, either $\ell_i \subseteq S$ or $\ell_i \cap H$ is a single point contained in $S$.
These constructions often satisfy the following stronger property, observed in \cite[Theorem 11]{fancsali2014lines}, which goes beyond the requirement that the collection of lines itself forms a strong blocking set.

\begin{definition}\label{def:line_avoidance}
    A set $\mathcal{L}$ of lines in $\mathrm{PG}(k - 1, q)$ is said to have the \textbf{avoidance property} if for all codimension-$2$ subspaces $W$, there exists a line $\ell \in \mathcal{L}$ such that $W \cap \ell = \emptyset$.
\end{definition}

\begin{theorem}\label{thm:avoid_prp}
If a set $\mathcal{L}$ of lines satisfies the avoidance property, then the point-set $\mathcal{B} = \cup_{\ell \in \mathcal{L}} \ell$ forms a strong blocking set. 
\end{theorem}
\begin{proof}
Let $\mathcal{L}$ be the set of lines and let $\mathcal{B} = \cup_{\ell \in \mathcal{L}}\ell$, and let us proof the contrapositive argument and suppose that $\mathcal{B}$ is not a strong blocking set. That means that there exists a hyperplane $H_{\Pi}$ such that $\mathcal{B} \cap H_{\Pi}$ does not span $H_{\Pi}$. In particular, $\mathcal{B} \cap H_{\Pi}$ is contained in a hyperplane $H$ of $H_{\Pi}$. Since we know that $H_{\Pi}$ is a hyperplane, it meets every line of the projective space. Thus $\ell \cap H_{\Pi} \neq \emptyset$ for all lines $\ell \in \mathcal{L}$, but since $\ell \subseteq \mathcal{B}$ and $\mathcal{B} \cap H_{\Pi} \subseteq H$, it follows that $\ell \cap H \neq \emptyset$. That is, $H$ is a codimension 2-subspace meeting every line $\ell \in \mathcal{L}$ and so that will imply $\mathcal{L}$ does not have the avoidance property.
\end{proof}

We will use this avoidance property in our construction as well. 

\subsection{Expander graphs}
Throughout our constructions, we will use explicit constructions of constant-degree expander graphs, as this general framework has been used by \cite{ABDN2023}. The informal way for someone to think about expander graphs is that the edges of those graphs are very well spread out, ensuring that there are many outgoing edges from all vertex subsets that are not too large. For further properties and applications we refer the reader to the survey \cite{hlw2006} for a formal definition and for a variety of applications of expander graphs in different areas of mathematics and computer science. 

Graph expansion can be studied by the spectral properties of the graph. Given an $n-$vertex graph $G$, we denote the eigenvalues of its adjacency matrix to be in the following decreasing order: $\lambda_{1} \geq \lambda_{2} \ldots \geq \lambda_{n}$. The eigenvalues of a graph encode substantial information about its underlying structure. In particular, the second-largest eigenvalue of the adjacency matrix provides insight into the connectivity and expansion properties of the graph.
For instance, If $G$ is a connected $t-$regular graph, then its largest eigenvalue $\lambda_{1} = t$, while the second-largest eigenvalue satisfies $\lambda_{2} < t$. 

\begin{definition}[Expander Graph]
A graph $G$ is called an $(n,t,\lambda)$-graph if it is a  $t$-regular connected graph on $n$ vertices with  $|\lambda_{i}| \leq \lambda$ for all $i \geq 2$.
\end{definition}

In our analysis, we will make use of  the \textit{Expander Mixing Lemma}. This lemma is often attributed to Alon and Chung \cite{alon1988explicit}, but it appears about 10 years earlier in the PhD thesis of Haemers \cite{Haemers1979}.

\begin{lemma}[Expander mixing lemma]\label{lem:exp_mix_tight}
Let $G$ be a $t$-regular $n$-vertex graph. Let $\lambda_{1} \geq \lambda_{2} \geq \lambda_{3} \geq \ldots \geq \lambda_{n}$ to be the eigenvalues of the adjacency matrix of $G$ and we define $\lambda = \max\{|\lambda_{2}|,|\lambda_{n}|\}$. Let $S,T \subset V(G)$ and disjoint. Denote by $e(S,T)$ to be the number of pairs $(x,y) \in S \times T$ such that $xy$ is an edge of $G$, then it holds,
$$\big|e(S,T)-\frac{t}{n}|S||T|\big| \leq \lambda \sqrt{|S||T|\Big(1-\frac{|S|}{n}\Big)\Big(1-\frac{|T|}{n}\Big)}$$ 
where $\lambda$ is the largest absolute value of an eigenvalue other than $\lambda_{1}$.
\end{lemma}
A detailed proof of the aforementioned lemma can be found in \cite[Lemma 4.15]{vadhan2012pseudorandomness}. 
It is worth noting that the error term on the right-hand side is directly proportional to $\lambda$, and it is natural to make this parameter as small as possible. The Alon--Boppana bound \cite{nilli1991second} places a fundamental limitation on how small ($\lambda$) can be. 
It states that for any infinite family of $(n, t, \lambda)$-graphs, $\lambda \geq 2 \sqrt{t - 1} - o(1)$ as $n \rightarrow \infty$.
This naturally leads to the definition of Ramanujan graphs, which constitute one of the most prominent families of expander graphs.

\begin{definition}
Let $G$ be a $t-$regular graph with eigenvalues $t = \lambda_{1} \geq \lambda_{2} \geq \ldots \geq \lambda_{n}$. If $\max\{|\lambda_{i}|: |\lambda_{i}| < t \} \leq 2\sqrt{t-1}$, then $G$ is said to be a \textbf{Ramanujan graph}.
\end{definition}

\begin{remark}
In the context of Ramanujan graphs, Lubotsky, Phillips and Sarnak ~\cite{lubotzky1988ramanujan} and Margulis ~\cite{margulis1988} gave explicit constructions of an $(n,p+1,2\sqrt{p})$ Ramanujan graphs for $n = \Theta(q^{3})$, where $p$ is a prime such that $p \equiv 1 \pmod{4}$ using linear algebraic groups.
\end{remark}
The above constructions can produce good $(n, d, \lambda)$-graphs with restricted range of values in $n$ and $d$. If we slightly relax the requirement that $\lambda \leq 2\sqrt{t-1}$, we can find explicit constructions for every degree d and every large enough n.

\begin{theorem}[\cite{alon2021explicit}, Theorem 1.3]\label{thm:bound_2}
For every positive integer $t$, and every $\varepsilon > 0$, there is an $n_{0}(t,\varepsilon)$ such that for all $n \geq n_{0}(t,\varepsilon)$ with $nt$ to be even, there is an explicit construction of an $(n,t,\lambda)-$graph $G^{\varepsilon}_{n,t}$ with $\lambda \leq 2\sqrt{t-1} + \varepsilon$
\end{theorem}

\section{New parameter and analysis}\label{sec:intro_param}
In \cite{ABDN2023}, the authors used the vertex integrity of a graph to obtain upper bounds on the size of a strong blocking set constructed from expander graphs. 
We first provide the definition of vertex integrity that measures the vulnerability of the graph under vertex deletions, essentially that means how easy it is to break the graph into small pieces.

\begin{definition}
Let $G = (V,E)$ be a simple connected graph. 
The \textit{integrity} of the graph $G$ is the integer 
$$\iota(G) \coloneqq \min_{{S \subseteq V}} \{|S| + \max_{D \in cc(G-S)}|D|\},$$
where $cc(G - S)$ is the set of connected components in the graph obtained by deleting $S$ from $G$.
\end{definition}

The new graph parameter as defined in \Cref{def:new_parameter} is motivated by the line avoidance property~\ref{def:line_avoidance} that we want to guarantee and the fact that the hyperplanes through any fixed codimension-$2$ subspace $S$ in $\mathrm{PG}(k-1, q)$  partition $\mathrm{PG}(k-1, q) \setminus S$ into $q + 1$ parts. 
If a collection $\mathcal{L}$ of lines does not have the avoidance property, then there exists a codimension-$2$ subspace $W$ such that all lines in our collection $\mathcal{L}$ intersect $W$.
Therefore, among the $q + 1$ hyperplanes $H_1, \dots, H_{q+1}$ containing $W$, there is no pair $H_i, H_j$ with a line $\ell \in \mathcal{L}$ that contains a point of $H_i \setminus W$ and a point of $H_j \setminus W$.
We define a graph parameter that captures this property and then we prove in \cref{lem:bound_param} that this new parameter is stronger than vertex integrity.

\begin{definition}\label{def:new_parameter}
    For a graph $G$ and integer $q \geq 2$, let 
    \[f_q(G) \coloneqq \min_{{S \subseteq V}} \{|S| + \max_i |C_i|: S \sqcup C_{1} \sqcup \cdots \sqcup C_{q+1} = V(G), e(C_i, C_j) = 0 ~\forall i \neq j\}.\]
\end{definition}

\begin{remark}
The definition is sensible for any integer $q \geq 1$, but for us $q$ will be a prime power and thus $q \geq 2$ is a fair assumption.
\end{remark}

\begin{lemma}\label{lem:bound_param}
    For every every graph $G$, 
    $f_q(G) \geq \iota(G)$. 
    
\end{lemma}
\begin{proof}
Take any subset $S \subseteq V(G)$ for which it holds that $f_{q}(G) = |S| + \max_{i=1,\ldots,q+1}|C_{i}|
$ and let $C_{1},C_{2},\ldots,C_{q+1}$ to be $q+1$ disjoint sets of size $|C_{i}| = c_{i}$ with no edges between them.  We can observe that for such a subset $S$, every connected component of $G -S$ must be contained in one of the $C_{i}$'s for $i = 1,2,\ldots,q+1$, hence the maximum size of any connected component of $G-S$ is at most $\max_{i = 1,\ldots,q+1}|C_{i}|$. Because the vertex integrity is the minimum over all subset $S$ we get that  $\iota(G) \leq |S| + \max_{i = 1,\ldots,q+1}|C_{i}| = f_{q}(G)$ and the desired inequality follows.
\end{proof}

In \cite{ABDN2023}, the authors established the following spectral lower bound in terms of vertex integrity.

\begin{lemma}
For any $(n,t,\lambda)$-graph $G$, we have that $\iota(G) \geq \left(\frac{t-\lambda}{t+\lambda}\right)n.$
\end{lemma}

The following lemma provides a spectral lower bound on $f_q(G)$, yielding a strict improvement over the bound obtained via vertex integrity.

\begin{lemma}\label[lemma]{lem:new_bound}
For any $(n,t,\lambda)$-graph $G$, we have that 
\begin{equation}\label{eq:main_bound}
    f_q(G) \geq  \left(\frac{t - \lambda + 2\lambda/q}{t + \lambda + 2\lambda/q}\right)n.
\end{equation}
\end{lemma}

\begin{proof}
Let $\mathcal{P} = \{S,C_{1},C_{2},\ldots,C_{q+1}\}$ be a partition of the vertex set of the graph $G$, $ V(G) = C_{1} \sqcup \cdots \sqcup C_{q+1}$, such that there are no edges between the components $C_{i}$ and $C_{j} \ \forall i \neq j$ and $f_q(G) = |S| + \max |C_i|$.
Without loss of generality, assume that the components have the following decreasing size ordering, $c_{1} \geq c_{2} \geq \ldots \geq c_{q+1}$ where $c_{i} = |C_{i}| \  \forall i = 1,2,\ldots,q+1$.
Therefore, $f_{q}(G) = s+c_{1}$, where $s = |S|$. 

We now give a lower bound on $s + c_1$. For the clarity of the analysis we define $C = \bigcup_{i = 1}^{q+1}C_{i}$ and $c = |C| =  \sum_{i=1}^{q+1} c_i = n - s$.
By \Cref{lem:exp_mix_tight}, we get
\begin{equation}
e(C) \geq \frac{tc^2}{2n} - \frac{\lambda}{2}c\left(1 - \frac{c}{n}\right)
= \frac{(t+\lambda)c^2}{2n} - \frac{\lambda c}{2}
\label{eq:edge-lower-bound}
\end{equation}
and
\begin{equation}
e(C_i) \leq \frac{(t-\lambda)c_i^2}{2n} + \frac{\lambda c_i}{2},
\label{eq:edge-upper-bound}
\end{equation}
for all $i = 1,2,\ldots,q+1$. 
The upper bounds on $e(C_i)$ imply
\begin{equation}
\begin{aligned}
e(C)
&\leq \sum_{i=1}^{q+1} e(C_i) \\
&\leq \frac{t-\lambda}{2n}\sum_{i=1}^{q+1} c_i^2
+ \frac{\lambda}{2}\sum_{i=1}^{q+1} c_i \\
&= \frac{t-\lambda}{2n}\sum_{i=1}^{q+1} c_i^2
+ \frac{\lambda c}{2}.
\end{aligned}
\label{eq:edge-upper-bound-sum}
\end{equation}
By combining \Cref{eq:edge-upper-bound-sum} and \Cref{eq:edge-lower-bound} we get  
\[\frac{(t+\lambda)c^2}{2n} - \frac{\lambda c}{2} \leq \frac{t - \lambda}{2n} \sum_{i=1}^{q+1} c_i^2 + \frac{\lambda c}{2}.\]
To simplify the right hand side, we use  $c_i^2 = c_{i}\cdot c_{i} \leq c_1 c_i$ to get $\sum_{i=1}^{q+1} c_i^2 \leq c_1 \sum_{i=1}^{q+1} c_i = c_1 c$.
Therefore, we get
\[(t + \lambda)c^2 - (t - \lambda)c_1 c \leq 2 \lambda n c.\]
Dividing by $c$, then substituting $c = n - s$, and rearranging the equation we get
\[(t + \lambda) s + (t - \lambda)c_1 \geq (t - \lambda) n.\]
Now substitute $s = f - c_1$ and we are ready to prove the lower bound on $f$, where $f = f_{q}(G)$. 
This gives 
\begin{equation}
(t + \lambda)f \geq 2\lambda c_1 + (t - \lambda)n.
\label{eq:bound-f}
\end{equation}
We now need a lower bound on $c_1$ in terms of the other parameters so that we can simplify this linear inequality into a lower bound on $f$. 
Note that $n - f = n - s - c_1 = \sum_{i = 2}^{q + 1} c_i \leq (q+1 -2 +1)c_1 = qc_{1}.$
This implies $c_1 \geq (n - f)/q$.
Substituting it to \Cref{eq:bound-f} we get
\[(t + \lambda)f \geq 2 \lambda (n - f)/q + (t - \lambda)n.\]
Solving for $f$ gives us the following lower bound,
\[
    f_q(G) = f \geq \left(\frac{t - \lambda + 2\lambda/q}{t + \lambda + 2\lambda/q}\right)n.
\]
\end{proof}

\vspace{1em}
Equation~\ref{eq:main_bound} gives a lower bound on the new parameter $f_q(G)$ which is \textit{strictly better} than the spectral lower bound on the vertex integrity of $G$ for all $\lambda > 0$. 
For example, for $q = 2$, we get 
$f_q(G) \geq  \left(\frac{t}{t + 2\lambda} \right) n,$
whereas the bound on vertex integrity is always 
$\iota(G) \geq \left( \frac{t - \lambda}{t + \lambda} \right)n.$

In the following section we compare the  improvements we get in the strong blocking set constructions using this new parameter.

\section{Constructing strong blocking sets from graphs}\label{sec:construction}
In this section, we give the graph-theoretic construction inspired from \cite{ABDN2023}, which we analyse using our new graph parameter.

\begin{definition}
Let $V=\{P_1,\ldots,P_n\}$ be a set of $n$ points in $\mathrm{PG}(k - 1, q)$ and let  $G=(V,E)$ be a  graph with vertex set equal to $V$. We define 
the set of lines
$$\mathcal{L}(V, G) \coloneqq \{ \langle P_i,P_j\rangle : P_iP_j \in E\}$$
and the set of points
$$\mathcal{B}(V,G) \coloneqq \bigcup_{\ell \in \mathcal{L}(V, G)} \ell $$
\end{definition}

\begin{remark} \label{rem:blockingsetsize}
    The size of $\mathcal{B}(V, G)$ is at most $n + (q - 1)|E|$, since there are $|E|$ lines, each of which contains at most $q - 1$ points not in $V$. 
\end{remark}

The following result lies at the heart of our construction as it gives a sufficient condition for the line-set $\mathcal{L}(V, G)$ to satisfy the avoidance property.

\begin{lemma}\label{lem:avoid_prop_param}
    Let $V \subseteq \mathrm{PG}(k - 1, q)$, with $|V| = n$ such that every hyperplane $H$ contains at most $n - d$ points of \  $V$, for some integer $d > 0$. 
    Let $G$ be a graph with vertex set $V(G) = V$ and edge set $E$, such that $f_q(G) \geq n - d + 1$. 
    Then $\mathcal{L} = \{ \langle x, y \rangle : xy \in E\}$ has the avoidance property.
\end{lemma}

\begin{proof}
Let us suppose for the sake of contradiction that the set of lines $\mathcal{L}$ does not satisfy the avoidance property.
Therefore, there exists a codimension-2 subspace $W$ such that every line $\ell \in \mathcal{L}$ contains a point of $W$.
Since $W$ is of codimension-2, there are exactly $q + 1$ distinct hyperplanes containing $W$, which partition the points of $\mathrm{PG}(k-1, q)$ outside $W$ into $q + 1$ parts.
Define $S = V \cap W$ and let  $C_{i} = (H_{i} \setminus W) \cap V$ for $i = 1,\ldots,q+1$, where $H_{i}$ is the $i$-th hyperplane through $W$ (see \cref{fig:visual_proof_avoid}).
Note that if a line contains two points of the subspace then it must be fully contained in it. 
Say, there exists an edge $e = xy$ with $x \in C_i$ and $y \in C_j$ for some $i \neq j$, and let $\ell = \langle x, y \rangle$.
Since $x \in H_i \setminus W$, $\ell$ is not fully contained in $W$.
If $\ell \cap W$ contains a point $z$, then $\ell$ must be fully contained in both $H_i$ and $H_j$, and thus $\ell \subseteq W = H_i \cap H_j$, which is a contradiction.
Therefore, there are no edges between any pair of components  $C_i$ and $C_{j}$ for $i\neq j$ and from the definition of $f_{q}(G)$, we get $f_{q}(G) \leq |S| + \max_{i \in [q+1]} |C_{i}|$. 

Since $S \cup C_{i} \subseteq H_{i}$, we have $|S| + |C_{i}| \leq n - d$ for all $i$. 
Therefore, $f_{q}(G) = |S| + \max_{i} |C_i| \leq n-d$, contradicting our assumption that $f_{q}(G) > n - d$.
\end{proof}

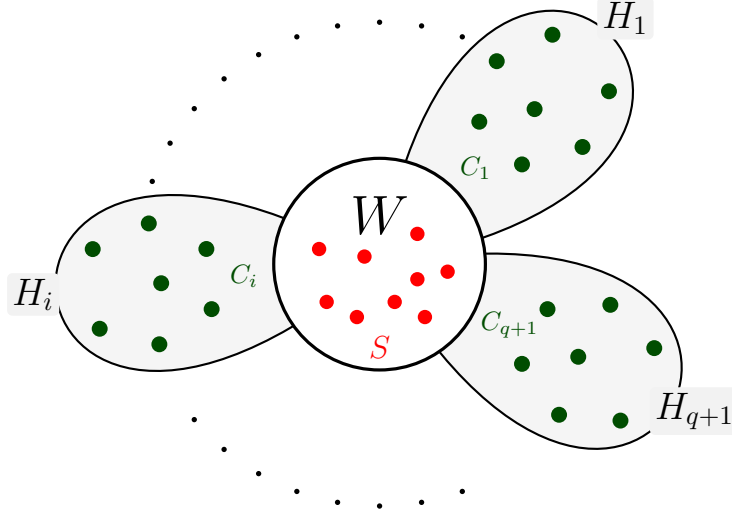
\begin{figure}[htb]
    \centering
\begin{tikzpicture}[
    scale=1,
    hyperplane/.style={thick, draw=black, fill=gray!10, fill opacity=0.7},
    labelstyle/.style={font=\Large\bfseries, text=black}
]

    \def\R{7}

    \foreach \ang/\text in {
        45/H_1,
        185/H_i,
        335/H_{q+1}
    } {
        
        \draw[hyperplane]
            (0,0) .. controls (\ang-35:\R) and (\ang+35:\R) .. (0,0);
    }

    \foreach \hang in {45, 185, 335} {
        \foreach \da/\r in {
            -10/2.3, 10/2.3,
            -15/3.1, 0/2.9, 15/3.1,
            -8/3.8, 8/3.8
        } {
            \fill[green!30!black] (\hang+\da:\r) circle (3.0pt);
        }
    }

    \node[green!30!black, font=\small\bfseries] at (45:1.8) {$C_1$};
    \node[green!30!black, font=\small\bfseries] at (185:1.8) {$C_i$};
    \node[green!30!black, font=\small\bfseries] at (335:1.9) {$C_{q+1}$};

    \foreach \ang/\text in {
        45/H_1,
        185/H_i,
        335/H_{q+1}
    } {
        \node[labelstyle, fill=gray!10, inner sep=2pt, rounded corners=2pt] at (\ang:4.6) {$\text$};
    }

    \foreach \ang in {70, 80, 90, 100, 110, 120, 130, 140, 150, 160,
                  220, 230, 240, 250, 260, 270, 280, 290} {
    \fill[black] (\ang:3.2) circle (1.0pt);
}

    \node[
        circle, 
        draw=black, 
        very thick, 
        fill=white, 
        minimum size=2.8cm, 
        inner sep=0pt
    ] (W) at (0,0) {};

    \node[font=\huge\bfseries] at (0, 0.65) {$W$};

    \foreach \coord in {
        (-0.8, 0.2), (-0.2, 0.1),
(0.5, 0.4), (-0.7, -0.5),
(0.2, -0.5), (0.5, -0.2),
(0.9, -0.1), (-0.3, -0.7),
(0.6, -0.7)
    } {
        \fill[red] \coord circle (2.7pt);
    }

    \node[red, font=\large\bfseries] at (0, -1.1) {$S$};

\end{tikzpicture}
 \caption{The figure represents the visualization of the construction in the proof  of \cref{lem:avoid_prop_param}, the circle in the middle represents the codimension-2 subspace $W$ and the petals of the flower represents the $q+1$ hyperplanes that goes through $W$ and partition the remaining points. The red points represents the set $S$ and the green points in every hyperplane represents those that belong to each one of the components $C_{i}$'s.}
    \label{fig:visual_proof_avoid}
\end{figure}

\begin{corollary}\label{cor:proj_avoid}
If $V$ is a projective $[n,k,d]_{q}$ system and $G = (V,E)$ is a graph with parameter $f_{q}(G) \geq n- d + 1$, then $\mathcal{B}(V,G)$ is a strong blocking set in $\mathrm{PG}(k-1,q)$ of size at most $n + (q - 1)|E|$.
\end{corollary}
\begin{proof}
A projective $[n, k, d]_q$ system corresponds to a set $V$ of points in $\mathrm{PG}(k-1, q)$ such that every hyperplane contains at most $n - d$ points of $V$. 
We saw that $\mathcal{L}(V, G)$ must be a set of $|E|$ lines that satisfy the avoidance property, and hence $\mathcal{B}(V, G)$ is a strong blocking set of size at most $n + (q-1)|E|$. 
\end{proof}

We proceed now to the proof of the main result by giving an explicit construction of strong blocking sets in $\mathrm{PG}(k-1,q)$ with size at most $C kq$, with a better constant than obtained in prior work.

\section{Explicit construction of strong blocking sets}\label{sec:ag_codes}

Following the approach of Alon et al.\cite{ABDN2023}, in combination with the new graph parameter that depends on the field size, we reduce the size of the explicit strong blocking sets, by using Algebraic-Geometric (AG) codes, expander graphs, and field reduction. 
In this section, we compute the new constant $C$ in the construction of a strong blocking set of size $C k (q + 1)$ and show that it is better in comparison with the previous constant in \cite{ABDN2023} for all values of $q$.

\subsection{Analysis of the new parameter over quadratic fields}
Our first result is obtained via a combination of \Cref{cor:proj_avoid} and \Cref{lem:new_bound} together with the expander graph explicit construction given by \cref{thm:bound_2}.
Note that for every square prime power $q$ there exists an explicit AG-code families of rates $R,\delta > 0$ that satisfy $R+ \delta \geq 1-(\sqrt{q}-1)^{-1}$ \cite{tsfasman2013algebraic, couvreur2021algebraic}, that is, we can construct an $(R,\delta)_{q}-$family of $[n_{i},k_{i},d_{i}]_{q}$ codes for some increasing sequences $\{n_{i}\}_{i \in \mathbb{N}},\{k_{i}\}_{i \in \mathbb{N}}$ and $\{d_{i}\}_{i \in \mathbb{N}}$.

\begin{theorem}\label{thm:quadratic_bound}
Let $t \geq 3$ and let $q > 4$ be a square prime power such that 
\[\frac{1}{\sqrt{q}-1} < \frac{t - 2\sqrt{t-1}+ \frac{4\sqrt{t-1}}{q}}{t + 2\sqrt{t-1} + \frac{4\sqrt{t-1}}{q}}\]
and let $\varepsilon > 0$. Then there is an increasing sequence $\{k_{i}\}_{i \in \mathbb{N}}$ for which we can explicitly construct strong blocking sets in $\mathrm{PG}(k_{i} - 1, q)$ of size at most $$\left(\frac{t(t+2\sqrt{t-1}+\frac{4\sqrt{t-1}}{q})(\sqrt{q}-1)}{2\left(t(\sqrt{q}-2) -2\sqrt{q(t-1)} + \frac{4(\sqrt{q(t-1)}-2\sqrt{t-1})}{q}\right)} + \varepsilon\right) k_{i}q.$$
\end{theorem}
Before analyzing the proof, we show what this result implies about the constant in the bound on the size of strong blocking set in $\mathrm{PG(k-1,q)}$ in the case where $q$ is a square. 
Then for each $q$, we choose an optimal value for $d$ to minimize the bound. 
This leads to finding the minimum values of the following function 
$$F_{q}(t) \coloneqq \frac{t(t+2\sqrt{t-1}+\frac{4\sqrt{t-1}}{q})(\sqrt{q}-1)}{2\left(t(\sqrt{q}-2) -2\sqrt{q(t-1)} + \frac{4(\sqrt{q(t-1)}-2\sqrt{t-1})}{q}\right)},$$
where we can extend the domain to $\mathbb{R}_{> 2}$. First to simplify the analysis we substitute $y = \sqrt{t-1}$ and then we want to find the local extrema by setting the derivative of the new created function to be equal to zero.
This amounts to finding the zeros of the polynomial: 
\[
\begin{aligned}
\psi_q(y)
={}&(\sqrt{q}-1)\Bigl[
 q^{5/2}\bigl(y^5-2y^4-2y^3+y+2\bigr)
 +8q^{3/2}\bigl(y^4+y^2\bigr)
  \\
&\quad -2q^{2}(y+1)(y^{2}+1)^{2}
+16\sqrt{q}\,y^{3}
 -16q\bigl(y^2+y+1\bigr)y^{2}
 -32y^{3}
\Bigr].
\end{aligned}
\]
We can see that because we are working in the domain where $(y>1)$ while $q$ grows the roots of this polynomial converge to those of $(y-1)(y+1)(y^{3}-2y^{2}-y-2)$ and the unique root will be 
$$y_{0} = \frac{1}{3}\left(2+(44-3\sqrt{177})^{1/3} + (44+3\sqrt{177})^{1/3})\right)$$
All the other roots will be either out of our domain $(y=-1,y=1)$ or they will be complex.
So for large values of $q$, then $F_{q}(t)$ will be minimized for 
$$t \approx t_{0} = 1+ y_{0}^{2} = 3 + \frac{1}{3}(459-12\sqrt{177})^{1/3} + \frac{1}{3}(459 + 12\sqrt{177})^{1/3} \approx 8.0701$$ 
so in that case someone should take either $t=8$ or $t = 9$ now for large enough $q$ we have that: 
$F_{q}(8) < F_{q}(9)$ and so while $q$ tends to infinity, the optimal constant that this construction provides is 
$$\lim_{q \rightarrow \infty} F_{q}(8) = \frac{4}{9}(23+8\sqrt{7}) \approx 19.63$$ because in the fraction that appears in $F_{q}(8)$ the leading $q$ terms cancel. 
In \cref{tab:square_primes} we can see for smaller values of $q$ the computation of optimal choice of $d$ and the corresponding constant in the upper bound. Motivated by the previous analysis we are ready to prove \cref{thm:quadratic_bound}.
\begin{proof}
Let $\varepsilon_{1} = \varepsilon_{1}(t,q) > 0$ be sufficiently small and we set $\lambda = 2\sqrt{t-1} + \varepsilon_{1}$. We have $(\sqrt{q} - 1)^{-1} < \frac{t-\lambda + \frac{2\lambda}{q}}{t + \lambda + \frac{2\lambda}{q}} - 2\varepsilon_{1}$ and we set also $R = \frac{t-\lambda + \frac{2\lambda}{q}}{t+ \lambda + \frac{2\lambda}{q}} - \varepsilon_{1} - (\sqrt{q} - 1)^{-1}$ and $\delta = 1-R - (\sqrt{q}-1)^{-1} = 1 - \frac{t-\lambda + \frac{2\lambda}{q}}{t+ \lambda + \frac{2\lambda}{q}} + \varepsilon_{1}$.

Let $\{\mathcal{A}_{n_{i},R}\}_{i \in \mathbb{N}}$ be the $(R,\delta)_{q}-$family of projective $[n_{i},k_{i},d_{i}]_{q}$ systems. \cref{thm:bound_2} shows that there exists some $K$ for which we obtain an explicit sequence $\{G_{i}\}_{i > K}$ of $(n_{i},d,\lambda)-$graphs. By the spectral bound of \cref{eq:main_bound} we have that $$f_{q}(G_{n_{i}}) \geq n_{i} \left(\frac{t - \lambda + 2\lambda/q}{t + \lambda + 2\lambda/q}\right) \geq n_{i}(1-\delta + \varepsilon_{1}).$$

Since $\lim_{i \rightarrow \infty} \frac{d_{i}}{n_{i}} = \delta$, we have $f_{q}(G_{n_{i}}) \geq n_{i} - d_{i} +1$ for sufficiently large $i$. Then from \Cref{cor:proj_avoid} $\mathcal{B}(\mathcal{A}_{n_{i},R},G_{n_{i}})$ is a strong blocking set in $\mathrm{PG}(k - 1, q)$. Since $G_{n_{i}}$ has $\frac{1}{2}n_{i}t$ edges by the handshaking lemma, we have that 
$$|\mathcal{B}(\mathcal{A}_{n_{i},R},G_{n_{i}})| \leq n_{i} + (q-1)\frac{n_{i}t}{2} < \frac{n_{i}t}{2}q$$
Now, since the $\lim_{i \rightarrow \infty} \frac{k_{i}}{n_{i}} = R$, there exists some $\varepsilon_{1} > 0$ such that $\frac{k_{i}}{n_{i}} \geq R - \varepsilon_{1} \implies \frac{n_{i}}{k_{i}} \leq \frac{1}{R-\varepsilon_{1}} \implies n_{i} \leq \frac{k_{i}}{R-\varepsilon_{1}}$ for sufficiently large $i$. After we make the substitution, and recall our choice of $R = \frac{t-\lambda + \frac{2\lambda}{q}}{t+ \lambda + \frac{2\lambda}{q}}- \varepsilon_{1} - (\sqrt{q} - 1)^{-1}$ so the upper bound becomes. 
$$\frac{n_{i}t}{2}q \leq \frac{t}{2(R-\varepsilon_{1})}k_{i}q = \frac{t(t+\lambda + \frac{2\lambda}{q})(\sqrt{q}-1)}{2(t-\lambda + \frac{2\lambda}{q})(\sqrt{q}-1) - 2(t+\lambda + \frac{2\lambda}{q}) - 4\varepsilon_{1}(t+ \lambda + \frac{2\lambda}{q})(\sqrt{q}-1)}k_{i}q$$. 

If we choose now $\varepsilon_{1}$ to be sufficiently small, we obtain the upper bound to be: 

$$|\mathcal{B}(\mathcal{A}_{n_{i},R},G_{n_{i}})| \leq \Big(\frac{t(t+\lambda + \frac{2\lambda}{q})(\sqrt{q}-1)}{2((t - \lambda + \frac{2\lambda}{q})(\sqrt{q}-1) - (t+\lambda + \frac{2\lambda}{q}))} + \frac{\varepsilon}{2}\Big)k_{i}q$$ and if we recall that $\lambda = 2\sqrt{t-1} + \varepsilon_{1}$ we get

$$\frac{t(t+\lambda + \frac{2\lambda}{q})(\sqrt{q}-1)}{2((t - \lambda + \frac{2\lambda}{q})(\sqrt{q}-1) - (t+\lambda + \frac{2\lambda}{q}))}=$$ 

$$\frac{t(t+2\sqrt{t-1} + \frac{4\sqrt{t-1}}{q})(\sqrt{q}-1) + \varepsilon_{1}t\Big((\sqrt{q}-1) + \frac{\sqrt{q}-1}{q}\Big)}{2\Big(t(\sqrt{q}-2) - 2\sqrt{q(t-1)} - \varepsilon_{1}\sqrt{q} + \frac{4}{q}(\sqrt{q(t-1)} + 2\varepsilon_{1}\sqrt{q} - 2\sqrt{t-1} - \varepsilon_{1})\Big)} $$
and the result follows provided that $\varepsilon_{1}$ is sufficiently small.
\end{proof}

\begin{center}
\begin{table}[htb]
    \centering
    \begin{tabular}{|c|c|c|c|}
        \hline
            $q$ & $\mathrm{argmin} \ F_{q}(t)$ & \text{Upper bound using} $f_{q}(G)$ & \text{Upper bound using} $\iota(G)$  \\ \hline
        9   & $t = 73$ & 264.63 & 292.68  \\ \hline
        16   & $t = 33$ & 96.02 & 104.60 \\  \hline
        25   & $t = 23$ & 62.60 & 66.86 \\ \hline
        49   & $t = 17$ & 42.17 & 43.91 \\ \hline
        64   & $t = 15$ & 37.82 & 39.07 \\ \hline
        81   & $t = 14$ & 34.88 & 35.83 \\ \hline
        121   & $t = 13$ & 31.15 & 31.76  \\ \hline
        169   & $t = 12$ & 28.89 & 29.31 \\ \hline
        $256 \leq q \leq 361$  & $t = 11$ & 26.79 & 27.06 \\ \hline
        $529 \leq q \leq 1024$   & $t = 10$ & 24.32 & 24.44 \\ \hline
        $1369 \leq q \leq 11881$   & $t = 9$ & 22.42 & 22.46 \\ \hline
        $q \geq 12769$   & $t = 8$  & 20.52 & 20.52 \\ \hline
    \end{tabular}
    \caption{For given ranges of square prime powers $q$, this table provides the degrees $t$ that minimize the size of strong blocking sets obtained by \cref{thm:quadratic_bound} and upper bounds the normalized constant $C$ in front of $k(q+1)$ on the corresponding sizes (third column)  and in fourth column we provide the previous bound in the normalized constant $C$ that was obtained using the graph parameter of vertex integrity.} 
    \label{tab:square_primes}

\end{table}
\end{center}

\subsection{Non-quadratic Fields}
\label{sec:non-quadratic}

In this section we show how to use the trick of field reduction to take a strong blocking set over the field $\mathbb F_{q^{2}}$ and build from it a strong blocking set over $\mathbb{F_{q}}$ that is not much larger. 
This is similar to the work done in \cite{ABDN2023}, that was based on \cite{alfarano2022three}, but we include it here for completeness.

We recall the \textit{Field reduction map}, which we denote by $\mathcal{F}_{q,r}$. This map uses the fact that points of $\mathrm{PG}(K-1,q^{r})$ are $1$-dimensional $\mathbb{F}_{q^r}$-subspaces of $\mathbb{F}^{K}_{q^{r}}$, which can be viewed as $r$-dimensional $\mathbb{F}_{q}-$subspaces of $\mathbb{F}^{rK}_{q}$. 
So the map $\mathcal{F}_{q,r}$ sends points of $\mathrm{PG}(K-1,q^{r})$ to $(r-1)$-spaces of $\mathrm{PG}(rK-1,q)$. For more details on the technique of field reduction we refer the reader to see \cite{lavrauw2015field}.

Also as can be shown in \cite[Theorem 4.2]{alfarano2022three}, the field reduction map preserves some key properties related to strong blocking sets and for that reason we begin the discussion by introducing the following definition.
\begin{definition}
Let $\mathcal{L} = \{\ell_{1},\ell_{2},\ldots,\ell_{t}\}$ be a collection of lines in $\mathrm{PG}(K-1,q^{2})$. We will say that a set of points $\Lambda \subseteq \mathrm{PG}(K-1,q^{2})$ is viable for $\mathcal{L}$ if $\Lambda = \cup_{i=1}^{t} \Lambda^{(i)}$, where each $\Lambda^{i} = \{\lambda_{1}^{(i)},\lambda_{2}^{(i)},\lambda_{3}^{(i)},\lambda_{4}^{(i)}\} \subseteq \ell_{i}$ is a set of four points that do not lie on a common $\mathbb{F}_{q}$-subline of $\ell_{i}$.

Given a viable set $\Lambda$, we define the derived set to be the set 
$$\mathcal{F}_{q,2}(\Lambda) =\{\mathcal{F}_{q,2}(\lambda_{j}^{(i)}): i \in [t],j \in [4]\}$$
of lines in $\mathrm{PG}(2K-1,q)$.
\end{definition}
\begin{theorem}(See \cite{alfarano2022three})\label{thm:viable_strong}
Let $\mathcal{L} = \{\ell_{1},\ell_{2},\ldots,\ell_{t}\}$ be a set of lines in $\mathrm{PG}(K-1,q^{2})$ whose union forms a strong blocking set. if $\Lambda$ is a viable set for $\mathcal{L}$, then the union of the lines in the derived set $\mathcal{F}_{q,2}(\Lambda)$ is a strong blocking set in $\mathrm{PG}(2K-1,q)$.
\end{theorem}
\begin{lemma}\label{lem:viable_exist}
Let $M$ be an $[n,K,d]_{q^{2}}$ projective system and let the graph $G = (M,E)$ be a graph. Then we can find a viable set $\Lambda$ of size at most $n+2|E|$ for the associated set of lines $\mathcal{L}(M,G)$ in $\mathrm{PG}(K-1,q^{2})$.
\end{lemma}

\begin{proof}
  Let us enumerate the edges of $G$ as $E = \{ e_1, e_2, \hdots, e_m \}$. If $e_i = P_a P_b$, then the corresponding line $\ell_i \in \mathcal{L}(M, G)$ is given by $\ell_i = \langle P_a, P_b \rangle$. We then take $\lambda^{(i)}_1 = P_a$ and $\lambda^{(i)}_2 = P_b$, and let $\lambda^{(i)}_3 = Q^{(a,b)}_3$ be an arbitrary third point on $\ell_i$. Since any three points on an $\mathbb{F}_{q^2}$-line define a unique $\mathbb{F}_q$-subline, we can then choose a fourth point $\lambda^{(i)}_4 = Q_4^{(a,b)}$ that avoids this subline.

Thus,
\[ \Lambda = \bigcup_{i=1}^m \left\{ \lambda^{(i)}_1, \lambda^{(i)}_2, \lambda^{(i)}_3, \lambda^{(i)}_4 \right\} = \mathcal{M} \cup \left( \bigcup_{P_a P_b \in E} \left\{ Q^{(a,b)}_3, Q^{(a,b)}_4 \right\} \right) \]
is viable for $\mathcal{L}(\mathcal{M},G)$, and $|\Lambda| \le n + 2|E|$. 
\end{proof}

We are ready now to apply \cref{lem:viable_exist} and  \cref{thm:viable_strong} to the construction of \cref{thm:quadratic_bound} in order to construct small strong blocking sets even when $q$ is not a square.
\begin{theorem}\label{thm:non_quadratic}
Let $t \geq 3$, and let $q > 2$ be such that $(q-1)^{-1} < \frac{t-2\sqrt{t-1}+ \frac{4\sqrt{t-1}}{q^{2}}}{t+ 2\sqrt{t-1} + \frac{4\sqrt{t-1}}{q^{2}}}$ and let $\varepsilon > 0$. Then there is an increasing sequence $\{k_{i}\}_{i \in \mathbb{N}}$ for which we can explicitly construct strong blocking sets in $\mathrm{PG}(k_{i}-1,q)$ of size at most 
$$\left(\frac{(t+1)(t+2\sqrt{t-1}+ \frac{4\sqrt{t-1}}{q^{2}})(q-1)}{2\Big(t(q-2) - 2q\sqrt{t-1} + \frac{4(q\sqrt{t-1} - 2\sqrt{t-1})}{q^{2}}\Big)}+\varepsilon \right)k_{i}(q+1)$$
\end{theorem}
\begin{proof}
The proof of this theorem follows the same direction as \cref{thm:quadratic_bound}, and so we will highlight the changes between them. As before we let $\varepsilon_{1} = \varepsilon_{1}(t,q,\lambda) > 0$ be sufficiently small and we set $\lambda = 2\sqrt{t-1} + \varepsilon_{1}$. We have $(q - 1)^{-1} < \frac{t-\lambda + \frac{2\lambda}{q^{2}}}{t + \lambda + \frac{2\lambda}{q^{2}}} - 2\varepsilon_{1}$ and we set also $R = \frac{t-\lambda + \frac{2\lambda}{q^{2}}}{t+ \lambda + \frac{2\lambda}{q^{2}}} - \varepsilon_{1} - (q - 1)^{-1}$ and $\delta = 1-R - (q-1)^{-1} = 1 - \frac{t-\lambda + \frac{2\lambda}{q^{2}}}{t+ \lambda + \frac{2\lambda}{q^{2}}} + \varepsilon_{1}$. \\ 
When we consider some $\{\mathcal{A}_{n_{i},R}\}_{i \in \mathbb{N}}$ to be an $(R,\delta)_{q^{2}}-$family of projective $[n_{i},K_{i},d_{i}]_{q^{2}}$ systems, and $\{G_{i}\}_{i > K}$ to be a sequence of $(n_{i},t,\lambda)$ graphs. Our choice of parameters ensures that if we take $i$ to be sufficiently large then we have that $f_{q}(G) \geq n_{i} -d_{i} + 1$.We conclude that the set of lines $\mathcal{L} = \mathcal{L}(\mathcal{A}_{n_{i},R},G)$ has the avoidance property in $\mathrm{PG}(K-1,q^{2})$ and hence by \cref{thm:avoid_prp} the union $\mathcal{B}(\mathcal{A}_{n_{i},R},G)$ of those lines forms a strong blocking set. Then by \cref{lem:viable_exist} we can find a set $\Lambda$ that is viable for $\mathcal{L}$ of size at most $|\mathcal{A}_{n_{i},R}| + 2e(G_{n_{i}}) = n_{i}(t+1).$ By \cref{thm:viable_strong} it shows that the union of the lines in the derived set $\mathbb{F}_{q,2}(\Lambda)$ is a strong blocking set in $\mathrm{PG}(2K_{i}-1,q)$. Because each line in $\mathrm{PG}(2K_{i}-1,q)$ has $q+1$ points, this strong blocking set has size at most $n_{i}(t+1)(q+1)$. We know as well that $R = \lim_{i \rightarrow \infty} \frac{K_{i}}{n_{i}}$, and recalling our choices for the parameters $R$ and $\lambda$ and setting $k_{i} = 2K_{i}$ then yields to the claimed bound.
\end{proof}
As before, we can determine the optimal degree $d$ to use by minimizing the following quantity: 
$$M_{q}(t) \coloneqq \frac{(t+1)(t+2\sqrt{t-1} + \frac{4\sqrt{t-1}}{q^{2}})(q-1)}{2\left(t(q-2) -2q\sqrt{t-1} + \frac{4(q\sqrt{t-1} - 2\sqrt{t-1})}{q^{2}}\right)}$$
We make now the substitution $t = 1+y^{2}$, following which we can find that the zeros of the derivative coincide with those of $\phi_{q}(y) = ((q-1)(q^{5}(y^{5}-2y^{4}-2y^{3}-2y^{2}+y+4)-2q^{4}(y^{5}+y^{4}+2y^{3}+y^{2}+y+2) +8q^{3}(y^{4}+y^{2})-16q^{2}y^{2}(y^{2}+y+1)+16qy^{3}-32y^{3})$.
When $q$ grows, the roots of $\phi_{q}(y)$ converge to the roots of $(y-1)(y+1)(y^{3}-2y^{2}-y-4)$.
This polynomial has a unique root $y_{0}$ because according to the restriction in our domain we cannot have $y= \pm 1$ and the root must be larger than 1, and that will correspond to 
$$t_{0} = 1+ y_{0}^{2} = 3 + (31-2\sqrt{58})^{1/3} + (31 + 2\sqrt{58})^{1/3} \approx 9.0967.$$
Hence, the asymptotically optimal degree must be either $d = 9$ or $d = 10$ and after some inspection in the values we will have that $M_{q}(10) > M_{q}(9) \rightarrow \frac{5}{49}(113+72\sqrt{2}) \approx 21.92$ while $q \rightarrow \infty$. 
We can see here that asymptotically for large values of $q$, theorem \cref{thm:non_quadratic} constructs a larger strong blocking set than \cref{thm:quadratic_bound}, also this theorem works over any field lifting the restriction being only quadratic, but it loses some value and slightly outerperform for small values of $q$, as we can see in the following table. We found out that field reduction method can lead to significantly smaller constants for small values of $q$.
Note that the process of derivation can be repeated by using maps $\mathcal{F}_{q, 2^h}$ for $h \geq 2$ as well. 

In the construction of a viable set, we choose four points from each line and  bound the number of lines in each iteration as being at most four times larger than the previous step. This will lead us to an explicit construction of strong blocking sets in $\mathrm{PG}(k_{i}-1,q)$ of size at most 
$$\left(\frac{2^{r-1}(t+1)(t+2\sqrt{t-1}+ \frac{4\sqrt{t-1}}{q^{2^{r+1}}})(q^{2^{r}}-1)}{t(q^{2^{r}}-2) - 2q^{2^{r}}\sqrt{t-1} + \frac{4(q^{2^{r}}\sqrt{t-1} - 2\sqrt{t-1})}{q^{2^{r+1}}}}+\varepsilon \right)k_{i}(q+1)$$
for some increasing sequence $\{k_{i}\}_{i \in \mathbb{N}}$ provided that $(q^{2^{r}}-1)^{-1} < \frac{t-2\sqrt{t-1}+ \frac{4\sqrt{t-1}}{q^{2^{r+1}}}}{t+ 2\sqrt{t-1} + \frac{4\sqrt{t-1}}{q^{2^{r+1}}}}$.
Throughout the performance of the optimization it was revealed that the third derivation $(r=3)$ minimizes the size of the blocking sets for $q = 2$ and that the second derivation $(r=2)$ is optimal for $3 \leq q \leq 5$.
We summarize all the results in \Cref{tab:my_table}, where we give the size of the smallest strong blocking sets obtained from our explicit construction.

\begin{center}
\begin{table}[htb]
    \centering
    \begin{tabular}{|c|c|c|c|c|}
        \hline
        $q$ & $\mathrm{argmin} \ M_{q}(t)$ & \text{Upper bound using} $f_{q}(G)$ & \text{Upper bound using} $\iota(G)$  \\ \hline
        3   & $t = 74 $ & 268.21 & 296.12  \\ \hline
        4   & $t = 34 $ & 98.88 & 107.35  \\ \hline
        5   & $t = 24 $ & 65.21 & 69.41  \\ \hline
        7   & $t = 18 $ & 44.61 & 46.32  \\ \hline
        8   & $t = 16 $ & 40.22 & 41.45 \\ \hline
        9   & $t = 15 $ & 37.25 & 38.18 \\ \hline
        11   & $t = 14 $ & 33.50 & 34.08\\ \hline
        13   & $t = 13 $ & 31.22 & 31.62\\ \hline
        $16 \leq q \leq 19$  & $t = 12$ & 28.60 & 29.36 \\ \hline
        $23 \leq q \leq 32$   & $t = 11$ & 26.62 & 26.73 \\ \hline
        $37 \leq q \leq 109$   & $t = 10$ & 24.70 & 24.75 \\ \hline
        $q \geq 113$   & $t = 9$  & 22.81& 22.81 \\ \hline 
    \end{tabular}
    \caption{For given ranges of prime powers $q$, this table provides the degrees $t$ that minimize the size of strong blocking sets that constructed in \cref{thm:non_quadratic}, and upper bounds the normalized constant $C$ in front of $k(q+1)$ on their corresponding sizes (third column) and in fourth column we provide the previous bound in the constant $C$ using the graph parameter of vertex integrity.}
    \label{tab:non_square_primes} 
\end{table}
\end{center}

\begin{center}
\begin{table}[htb]
    \centering
    \begin{tabular}{|c|c|c|c|c|}
        \hline
        $q$ & Construction & \text{Upper bound using} $f_{q}(G)$ & \text{Upper bound using} $\iota(G)$    \\ \hline
        2   & 3rd derivation  & 116.42 & 117.46     \\   \hline
        3   & 2nd derivation   & 74.50 & 76.38     \\ \hline
        4   & 2nd derivation & 58.21 & 58.73     \\ \hline
        5   & 2nd derivation  & 52.39 & 52.59   \\ \hline
        7   & 1st derivation  & 44.61 & 46.31    \\ \hline
        8   & 1st derivation  & 40.22 & 41.45    \\ \hline
        9   & 1st derivation  & 37.25 & 38.19   \\ \hline
        11   & 1st derivation  & 33.50 & 34.08    \\ \hline
        13   & 1st derivation  & 31.22 & 31.62   \\ \hline
        16   & 1st derivation  & 29.11 & 29.36  \\ \hline
        17   & 1st derivation & 28.60 &28.82    \\ \hline
        19   & 1st  derivation & 27.77 & 27.96   \\ \hline
        $23 \leq q \leq 25$  &  1st derivation  & 26.19 & 26.29  \\ \hline
        $27 \leq q \leq 32$   & 1st derivation & 25.20 & 25.25  \\ \hline
        $37 \leq q \leq 49$   & 1st derivation & 23.98 & 24.00 \\ \hline
        $53 \leq q \leq 109$   & 1st derivation & 22.84 &  22.85   \\ \hline
        $113 \leq q \leq 1217$   & 1st derivation & 22.00 &  22.00 \\ \hline
        $1223 \leq q \leq 12763$   & 1st derivation & 21.93 & 21.93   \\ \hline   
        $q \geq 12769$, $q$ non-square   & 1st derivation  & 21.93 & 21.93   \\ \hline 
        $12769\leq q \leq 70603$, $q$ square  & original  & 20.00 & 20.06  \\ \hline 
        $q > 70603$, $q$ square  &  original  & 19.98 & 20.00  \\ \hline 
        
    \end{tabular}
    \caption{For given ranges of  prime powers $q$, this table provides the best upper bound on the size of strong blocking sets in finite projective spaces over $\mathbb{F}_{q}$ that we have obtained in the normalized constant $C$ in front of $k(q+1)$ (third column)  together with an indication of which construction achieve this: Original ($\cref{thm:quadratic_bound}$), 1st derivation ($\cref{thm:non_quadratic}$) or $r^{\textit{th}}$ derivation.
    The fourth columns contains the earlier bounds of \cite{ABDN2023} that used vertex integrity of a graph.}
    \label{tab:my_table}
\end{table}
\end{center}

\section{Conclusion}\label{sec:conclusion}
We introduced a new graph parameter that incorporates the size of the finite field and used it to give improved explicit constructions of strong blocking sets. 
Specifically, we provide a lower bound on this new graph parameter for the family of $d-$regular graphs that depends on both the eigenvalues of the adjacency matrix and the size of the finite filed $\mathbb{F}_{q}$, this deviating from the earlier work \cite{ABDN2023} that relied on a graph parameter which was independent of the field size.

In a recent preprint, Chen \cite{chen2026simpleexplicitconstructionsmaller} has improved the explicit construction of strong blocking sets in $\mathrm{PG}(k -1, q)$ for all $q \leq 11$ using concatenation of codes. 
For $q \geq 13$, however, the strong blocking sets from \Cref{thm:non_quadratic} and \Cref{thm:quadratic_bound} are better than this and all previous constructions \cite{bartoli2023small,cohen2016yet,chen2026simpleexplicitconstructionsmaller,ABDN2023}. 

Several directions for future work remain.
First, it would be nice to find a graph theoretic construction of strong blocking sets that does not require the stronger line avoidance property,
Such a construction might involve a new graph parameter that  further optimizes the constant and leads to a smaller explicit constructions of strong blocking sets.
Second, motivated by \cite{bishnoi2026explicit}, it would be interesting to extend our constructions based on the new graph parameter to strong blocking sets with respect to codimension-$2$ spaces.
Specifically, we would like to define and analyze similar parameter for $3$-uniform hypergraphs which could improve the bounds from \cite{bishnoi2026explicit}.
Ultimately, one can try generalizing our constructions to strong $s$-blocking sets. 
It would be interesting to match, or improve, the recent bounds of Guo et al.~\cite{guo2026explicit} using these graph-theoretic techniques on strong $s$-blocking sets.

\section*{Acknowledgments}
The work of the first author has been supported by  the Dutch Research Council (NWO) through the NWO Talent Programme Vidi project ``Extremal problems in finite geometry'' (project number VI.Vidi.243.039).\\

\noindent
The work of the second author has been partially supported by the MEYS of the Czech Republic
under the INTER-EXCELLENCE II program (project No. LUAUS26294), the European Union under the project Robotics and advanced industrial production (reg. no. CZ.02.01.01/00/22\_008/0004590) and also supported by the Grant Agency of the Czech Technical University in Prague, grant No.SGS26/181/OHK3/3T/18.
The research visit of the second author to TU Delft was supported by the Vidi grant of the first author.

\subsection*{Statement on software use}
The authors did not use AI for any mathematical proofs or writing. 
WolframAlpha was used for finding minima for large values of $q$ in \Cref{tab:non_square_primes,tab:square_primes} and for searching for optimal values of $r$ in \Cref{tab:my_table}. 
Freely available generative AI tools were used to create an example of \Cref{fig:visual_proof_avoid} and then the final modification was done by the authors. 
Finally, the authors used common search engines that by default use AI, for exploring relevant literature.

\printbibliography

@article{alfarano2022geometric,
  author = {Alfarano, G. N. and Borello, M. and Neri, A.},
  title = {A geometric characterization of minimal codes and their asymptotic performance},
  journal = {Advances in Mathematics of Communications},
  volume = {16},
  number = {1},
  pages = {115--133},
  year = {2022}
}

@article{alfarano2022three,
  author = {Alfarano, G. N. and Borello, M. and Neri, A. and Ravagnani, A.},
  title = {Three combinatorial perspectives on minimal codes},
  journal = {SIAM Journal on Discrete Mathematics},
  volume = {36},
  number = {1},
  pages = {461--489},
  year = {2022}
}

@article{alon1988explicit,
  author = {Alon, N. and Chung, F. K.},
  title = {Explicit construction of linear sized tolerant networks},
  journal = {Discrete Mathematics},
  volume = {72},
  number = {1--3},
  pages = {15--19},
  year = {1988}
}

@article{ashikhmin1998minimal,
  author = {Ashikhmin, A. and Barg, A.},
  title = {Minimal vectors in linear codes},
  journal = {IEEE Transactions on Information Theory},
  volume = {44},
  number = {5},
  pages = {2010--2017},
  year = {1998}
}

@article{bartoli2023small,
  author = {Bartoli, Daniele and Borello, Martino},
  title = {Small strong blocking sets by concatenation},
  journal = {SIAM Journal on Discrete Mathematics},
  volume = {37},
  number = {1},
  pages = {65--82},
  year = {2023}
}

@article{davydov2011linear,
  author = {Davydov, A. A. and Giulietti, M. and Marcugini, S. and Pambianco, F.},
  title = {Linear nonbinary covering codes and saturating sets in projective spaces},
  journal = {Advances in Mathematics of Communications},
  volume = {5},
  number = {1},
  pages = {119--147},
  year = {2011}
}

@article{fancsali2014lines,
  author = {Fancsali, S. and Sziklai, P.},
  title = {Lines in higgledy-piggledy arrangement},
  journal = {The Electronic Journal of Combinatorics},
  volume = {21},
  number = {2},
  pages = {P2.56},
  year = {2014}
}

@phdthesis{Haemers1979,
  author = {W. H. Haemers},
  title  = {Eigenvalue Techniques in Design and Graph Theory},
  school = {Eindhoven University of Technology},
  year   = {1979},
}

@book{vanLint_book,
  author = {van Lint, J. H.},
  title = {Introduction to Coding Theory},
  series = {Graduate Texts in Mathematics},
  volume = {86},
  publisher = {Springer-Verlag},
  address = {Berlin, Heidelberg},
  year = {1982},
}

@article{guo2026explicit,
  author  = {Guo, Zeyu and Raj, Roshan and Shangguan, Chong and Zhang, Zihan},
  title   = {Explicit Rank Extractors and Subspace Designs via Function Fields, with Applications to Strong Blocking Sets},
  journal = {arXiv preprint arXiv:2604.13431},
  year    = {2026},
}

@misc{coding_theory_book,
  author = {Guruswami, V. and Rudra, A. and Sudan, M.},
  title = {Essential coding theory},
  year = {2022},
  note = {Draft},
  url = {https://cse.buffalo.edu/faculty/atri/courses/coding-theory/book/}
}

@article{heger2021short,
  author = {H{\'e}ger, T. and Nagy, Z. L.},
  title = {Short minimal codes and covering codes via strong blocking sets in projective spaces},
  journal = {IEEE Transactions on Information Theory},
  volume = {68},
  number = {2},
  pages = {881--890},
  year = {2022}
}

@article{heger2015search,
  author = {H{\'e}ger, T. and Patk{\'o}s, B. and Tak{\'a}ts, M.},
  title = {Search problems in vector spaces},
  journal = {Designs, Codes and Cryptography},
  volume = {76},
  number = {2},
  pages = {207--216},
  year = {2015}
}

@article{hlw2006,
  author = {Hoory, S. and Linial, N. and Wigderson, A.},
  title = {Expander graphs and their applications},
  journal = {Bulletin of the American Mathematical Society},
  volume = {43},
  number = {4},
  pages = {439--561},
  year = {2006}
}

@article{lubotzky1988ramanujan,
  author = {Lubotzky, A. and Phillips, R. and Sarnak, P.},
  title = {Ramanujan graphs},
  journal = {Combinatorica},
  volume = {8},
  number = {3},
  pages = {261--277},
  year = {1988}
}

@inproceedings{massey1993minimal,
  author = {Massey, J. L.},
  title = {Minimal codewords and secret sharing},
  booktitle = {Proceedings of the 6th joint Swedish-Russian international workshop on information theory},
  pages = {276--279},
  year = {1993}
}

@article{miklos1984linear,
  author = {Mikl{\'o}s, D},
  title = {Linear binary codes with intersection properties},
  journal = {Discrete Applied Mathematics},
  volume = {9},
  number = {2},
  pages = {187--196},
  year = {1984}
}

@article{tang2021full,
  author = {Tang, C. and Qiu, Y. and Liao, Q. and Zhou, Z.},
  title = {Full characterization of minimal linear codes as cutting blocking sets},
  journal = {IEEE Transactions on Information Theory},
  volume = {67},
  number = {6},
  pages = {3690--3700},
  year = {2021}
}

@book{tsfasman2013algebraic,
  author = {Tsfasman, M. and Vladut, S. G.},
  title = {Algebraic-geometric codes},
  volume = {58},
  publisher = {Springer Science \& Business Media},
  year = {2013}
}

@inproceedings{ta2017explicit,
  author = {Ta-Shma, Amnon},
  title = {Explicit, almost optimal, epsilon-balanced codes},
  booktitle = {Proceedings of the 49th Annual ACM SIGACT Symposium on Theory of Computing},
  pages = {238--251},
  year = {2017}
}

@article{tsfasman1982modular,
  author = {Tsfasman, Michael A and Vl{\u{a}}dut, S. G. and Zink, T.},
  title = {Modular curves, {S}himura curves, and {G}oppa codes, better than the {V}arshamov--{G}ilbert bound},
  journal = {Mathematische Nachrichten},
  volume = {109},
  number = {1},
  pages = {21--28},
  year = {1982}
}

@article{margulis1988,
  author = {Margulis, G. A.},
  title = {Explicit group-theoretic constructions of combinatorial schemes and their applications in the construction of expanders and concentrators},
  journal = {Problemy Peredachi Informatsii},
  volume = {24},
  number = {1},
  pages = {51--60},
  year = {1988}
}

@article{alon2021explicit,
  author = {Alon, Noga},
  title = {Explicit expanders of every degree and size},
  journal = {Combinatorica},
  volume = {41},
  number = {4},
  pages = {447--463},
  year = {2021}
}

@incollection{lavrauw2015field,
  author = {Lavrauw, M. and Van de Voorde, G.},
  title = {Field reduction and linear sets in finite geometry},
  booktitle = {Topics in Finite Fields},
  editor = {Kyureghyan, G. and Mullen, G. L. and Pott, A.},
  series = {Contemporary Mathematics},
  volume = {632},
  pages = {271--293},
  publisher = {American Mathematical Society},
  year = {2015}
}

@article{bartoli2022cutting,
  author = {Bartoli, Daniele and Cossidente, Antonio and Marino, Giuseppe and Pavese, Francesco},
  title = {On cutting blocking sets and their codes},
  journal = {Forum Mathematicum},
  volume = {34},
  number = {2},
  pages = {347--368},
  year = {2022}
}

@article{bonini2021minimal,
  author = {Bonini, Matteo and Borello, Martino},
  title = {Minimal linear codes arising from blocking sets},
  journal = {Journal of Algebraic Combinatorics},
  volume = {53},
  pages = {327--341},
  year = {2021}
}

@article{Bishnoi2023trifference,
  author = {Bishnoi, A. and D'haeseleer, J. and Gijswijt, D. and Potukuchi, A.},
  title = {Blocking sets, minimal codes and trifferent codes},
  journal = {Journal of the London Mathematical Society},
  volume = {109},
  number = {6},
  pages = {e12938},
  year = {2024}
}

@article{ABDN2023,
  author = {Alon, N. and Bishnoi, A. and Das, S. and Neri, A.},
  title = {Strong blocking sets and minimal codes from expander graphs},
  journal = {Transactions of the American Mathematical Society},
  volume = {377},
  number = {8},
  pages = {5389--5410},
  year = {2024}
}

@article{ABN2023,
  author = {Alfarano, G. N. and Borello, M. and Neri, A.},
  title = {Outer strong blocking sets},
  journal = {The Electronic Journal of Combinatorics},
  volume = {31},
  number = {2},
  pages = {P2.18},
  year = {2024}
}

@article{vadhan2012pseudorandomness,
  author = {Vadhan, S. P.},
  title = {Pseudorandomness},
  journal = {Foundations and Trends in Theoretical Computer Science},
  volume = {7},
  number = {1--3},
  pages = {1--336},
  year = {2012}
}

@incollection{couvreur2021algebraic,
  author = {Couvreur, Alain and Randriambololona, Hugues},
  title = {Algebraic geometry codes and some applications},
  booktitle = {Concise Encyclopedia of Coding Theory},
  editor = {Huffman, W. C. and Kim, J.-L. and Sol{\'e}, P.},
  pages = {307--362},
  publisher = {Chapman and Hall/CRC},
  year = {2021}
}

@article{justesen1972class,
  author = {Justesen, J{\o}rn},
  title = {Class of constructive asymptotically good algebraic codes},
  journal = {IEEE Transactions on Information Theory},
  volume = {18},
  number = {5},
  pages = {652--656},
  year = {1972}
}

@article{cohen2016yet,
  author = {Cohen, G. and Mesnager, S. and Randriam, H.},
  title = {Yet another variation on minimal linear codes},
  journal = {Advances in Mathematics of Communications},
  volume = {10},
  number = {1},
  pages = {53--61},
  year = {2016}
}

@article{nilli1991second,
  author = {Nilli, A.},
  title = {On the second eigenvalue of a graph},
  journal = {Discrete Mathematics},
  volume = {91},
  number = {2},
  pages = {207--210},
  year = {1991}
}

@article{scotti2024lower,
  title={On the lower bound for the length of minimal codes},
  author={Scotti, Martin},
  journal={Discrete Mathematics},
  volume={347},
  number={1},
  pages={113676},
  year={2024},
  publisher={Elsevier}
}

@article{bishnoi2026explicit,
  author = {Bishnoi, A. and Tomon, Istv{\'a}n},
  title = {Explicit constructions of optimal blocking sets and minimal codes},
  journal = {Combinatorica},
  volume = {46},
  number = {2},
  pages = {13},
  year = {2026}
}

@article{lia2025note,
  author = {Lia, S. and Van de Voorde, G.},
  title = {A note on strong blocking sets and higgledy-piggledy sets of lines},
  journal = {The Electronic Journal of Combinatorics},
  volume = {32},
  number = {3},
  pages = {P3.27},
  year = {2025},
}

@article{borello2025geometry,
  author = {Borello, M. and Schmid, W. and Scotti, M.},
  title = {The geometry of intersecting codes and applications to additive combinatorics and factorization theory},
  journal = {Journal of Combinatorial Theory, Series A},
  volume = {214},
  pages = {106023},
  year = {2025},
}

@misc{chen2026simpleexplicitconstructionsmaller,
      title={A Simple Explicit Construction of Smaller Strong Blocking Sets}, 
      author={Hao Chen},
      year={2026},
      eprint={2610.12819},
      archivePrefix={arXiv},
      primaryClass={math.CO},
      
}

\end{document}